\documentclass[a4paper, 12pt]{amsart}

\usepackage{amssymb,amsmath,amsthm,amscd}

\usepackage{amscd}
\usepackage{mathrsfs}
\usepackage{color}
\usepackage[all]{xy}

\def\today{\number\day\space\ifcase\month\or   January\or February\or
   March\or April\or May\or June\or   July\or August\or September\or
   October\or November\or December\fi\   \number\year}

\theoremstyle{definition}
\newtheorem{thm}{Theorem}[section]
\newtheorem{lem}[thm]{Lemma}
\newtheorem{prp}[thm]{Proposition}
\newtheorem{dfn}[thm]{Definition}
\newtheorem{cor}[thm]{Corollary}

\newtheorem{rmk}[thm]{Remark}

\newtheorem{pbm}[thm]{Problem}

\def\s{\sharp}
\def\na{\natural}
\newcommand{\UIN}[1]{\left|\!\left|\!\left|{#1}\right|\!\right|\!\right|}
\newcommand{\NORM}[1]{\left|\!\left|{#1}\right|\!\right|}

{\normalsize }
\def\u|{|\kern-0.1em|\kern-0.1em|}
\def\U|{\Big|\kern-0.1em\Big|\kern-0.1em\Big|}

\def\~{\hskip-2pt}

\def\<{\langle}
\def\>{\rangle}

\def\VV{\lower-0.1ex\hbox{$\ \begin{matrix}\vee\\[-2ex]\vee\end{matrix}\ $}}
\def\vv{\lower-0.2ex\hbox{$\ \begin{matrix}\wedge\\[-2ex]\wedge\end{matrix}\ $}}

\def\({\left(}
\def\){\right)}

\def\l{\lambda}
\def\u|{|\kern-0.1em|\kern-0.1em|}
\def\U|{\Big|\kern-0.1em\Big|\kern-0.1em\Big|}

\numberwithin{equation}{section}
\def\s{\sharp}
\def\l{\lambda}
\def\na{\natural}

\def\phi{\varphi}

\title[On the Ando-Hiai property]{On the Ando-Hiai property for spectral geometric means}

\author[Y. Seo]{Yuki Seo}
\address{Department of Mathematics Education, Osaka Kyoiku University, Asahigaoka, Kashiwara, Osaka582-8582, Japan}
\email[]{yukis@cc.osaka-kyoiku.ac.jp}

\author[S. Wada]{Shuhei Wada}
\address{Department of Information and Computer Engineering
National Institute of Technology, Kisarazu College
2-11-1 Kiyomidai-Higasi
Kisarazu Chiba 292-0041
Japan}
\email[]{wada@j.kisarazu.ac.jp}

\author[T. Yamazaki]{Takeaki Yamazaki}
\address{Department of Electrical, Electronic and Communications Engineering, Toyo University, Kawagoe-Shi, Saitama, 350-8585, Japan}
\email[]{t-yamazaki@toyo.jp}

\date{\today}



\begin{document}

\maketitle

\begin{abstract}
In this paper, we consider a two-variable operator function that includes two weighted spectral geometric means, and show fundamental properties of the operator function. Moreover, it satisfies the Ando-Hiai type inequality under some restricted conditions. As an application, we show the log-majorization relations and norm inequalities for the spectral geometric means of positive definite matrices. 
\end{abstract}

\section{Introduction}

 Let $H$ be a complex Hilbert space, and $\mathcal{B}(H)$ be the $C^*$-algebra of all bounded linear operators  {on $H$}. An operator $T\in \mathcal{B}(H)$ is said to be positive if $\langle Tx,x\rangle \geq 0$ holds for all $x\in H$, and we write $T\in \mathcal{B}(H)_{+}$ or $T\geq 0$. If in addition $T$ is invertible, we write $T\in \mathcal{B}(H)_{++}$ or $T>0$. 
For self-adjoint $A,B\in \mathcal{B}(H)$, if $A-B\in \mathcal{B}(H)_{+}$, then we write $A\geq B$. This is known as the L\"{o}wner partial order. \par 

For $A\in \mathcal{B}(H)_{++}$ and $B\in \mathcal{B}(H)_{+}$, the (weighted) geometric mean is defined by
\[
A \#_{t} B = A^{\frac{1}{2}}(A^{-\frac{1}{2}}BA^{-\frac{1}{2}})^t A^{\frac{1}{2}}, \qquad t\in [0,1],
\]
see \cite{KA}. In the case of $t=\frac{1}{2}$, we denote $A \#_{\frac{1}{2}} B$ by $A \# B$ for simplicity. The geometric mean $A \#_{t} B$ satisfies many desiblack properties as a geometric mean such as (i) Transposition: $A \#_t B = B \#_{1-t} A$. (ii) Consistency with scalars: If $A$ and $B$ commute, then $A \#_t B = A^{1-t}B^t$. (iii) Homogeneity: $(\l A) \#_t (\l B) =\l (A \#_t B)$ for $\l>0$. (iv) Unitary invariance: $(U^*AU) \#_t (U^*BU) = U^*(A \#_t B)U$ for any unitary $U$. (v) Self-duality: $(A \#_t B)^{-1} = A^{-1} \#_t B^{-1}$. (vi) Arithmetic-Geometric-Harmonic mean inequality: $
\left((1-t)A^{-1}+tB^{-1}\right)^{-1} \leq A \#_t B \leq (1-t)A+tB$, see also \cite[Lemma 3.2]{FMPS2}.
\par
On the other hand, for $A\in \mathcal{B}(H)_{++}$ and $B\in \mathcal{B}(H)_{+}$, Fiedler-Pet\'{a}k \cite{FP}  introduced
\[
A \natural B := (A^{-1}\#B)^{\frac{1}{2}}A(A^{-1}\#B)^{\frac{1}{2}},
\]
and they showed that for positive definite matrices $A$ and $B$, $(A \natural B)^2$ is similar to $AB$ and thus the eigenvalues of $A \na B$ are the positive square roots of the eigenvalues of $AB$. 
 {Hence $A\natural B$ is called the {\it spectral geometric mean}.}
As a generalization of this operation, the following family was defined in \cite{Lee and Lim}:
\[
A \natural_t B := (A^{-1}\#B)^t A(A^{-1}\#B)^t, \qquad t\in[0,1].
\]
Later, in \cite{Dinh-Tam-Vuong}, the following another extension was proposed:
\[
A \widetilde{\natural}_{k } B := (A^{-1}\#_k B)^{\frac{1}{2}}\,A^{2(1-k)}\,(A^{-1}\#_k B)^{\frac{1}{2}}, \qquad k\in[0,1].
\]
Obviously, $A\natural_{\frac{1}{2}}B=A\widetilde{\natural_{\frac{1}{2}}}B=A\natural B$  {holds}, and they are called the {\it weighted spectral geometric means}.
These operations satisfy several properties that are naturally expected from (weighted) geometric means.

 In \cite{AH1994}, Ando and Hiai proposed a log-majorization relation, whose essential part is the following operator inequality (AH). We say it the Ando-Hiai inequality: Let $A, B\in \mathcal{B}(H)_{++}$ and $t\in [0,1]$.
 Then
\begin{equation} \label{eq:AH}
A \#_t B \leq I \quad \mbox{implies} \quad A^{ {p}} \#_t B^{ {p}} \leq I \quad \mbox{for all ${ {p}}\geq 1$}. \tag{AH}
\end{equation} 
It is known in \cite{FK2006} that the Ando-Hiai inequality (AH) is equivalent to the Furuta inequality  {\cite{FI}}. Moreover, it follows from (AH) and the Lie-Trotter formula for $\#_t$ that 
\begin{equation} \label{eq:se}
\NORM{(A^p \#_t B^p)^{\frac{1}{p}}} \ \nearrow \ \NORM{\exp((1-t)\log A+t \log B)} \quad \mbox{as $p \downarrow 0$}
\end{equation}
for the operator norm $\NORM{\cdot}$, see \cite{AH1994}.\par
Motivated by the Ando-Hiai inequality (AH) for the weighted geometric mean, Gan-Tam \cite[Theorem 3.3]{GT} studied the Ando-Hiai type inequality for the weighted spectral geometric mean $\natural_{t}$. However, there is a gap in that proof, and it is still unclear whether the result obtained by Gan-Tam holds or not. Thus, in this paper, we would like to consider the Ando-Hiai type inequality for the weighted spectral geometric mean under a more general framework. In section 2, we define a two-variable operator function that includes the weighted spectral geometric means $A \natural_t B$ and $A \widetilde{\natural}_{k } B$ and show their fundamental properties. In section 3, we show  {an} Ando-Hiai type inequality for a two-variable operator function under some restricted conditions. As applications of the  {result}, we show Ando-Hiai  {inequalities} for the weighted spectral geometric means $A \natural_t B$ and $A \widetilde{\natural}_{k } B$ under restricted conditions. Especially, we show that $A\widetilde{\natural_{k}}B$ does not satisfy the Ando-Hiai type inequality for $k\in (\frac{1}{2},1)$.  {In section 4, we obtain a generalization of the Ando-Hiai inequality which is consideblack two positive parameters on the exponents.}
In sections  {5 and 6}, we shall deal with matrices, and we show the log-majorization relations and norm inequalities for the spectral geometric means of positive definite matrices. Finally, we show an eigenvalues inequality of the alternative means, and as a corollary we have a weak log-majorization relation of the alternative means.

\section{Two-Variable Operator Functions and Binary Operations
}

To discuss the Ando-Hiai type inequalities for the weighted spectral geometric means, we would like to consider a slightly more general framework that includes the spectral geometric means.\par

In this section, first of all, we define the following  {operator function with three parameters} that includes the spectral geometric means $\natural_t$ and $\widetilde{\natural}_{k}$: 
For $A\in \mathcal{B}(H)_{++}$, $B\in \mathcal{B}(H)_{+}$, real parameters $k ,t \in (0,1)$ and $L >0$, define
\[
F_{k ,t ,L }(A,B):=
(A^{-1}\#_k  B)^{t }\,A^{L }\,(A^{-1}\#_k  B)^{t }.
\]
It follows that $F_{k ,t ,L }$ exhibits 
geometric-mean-like behavior for certain choices of $k ,t ,L $, that is, this function coincides with the binary operation $\natural_t$ when $(k,t,L)=(\tfrac{1}{2},t,1)$, 
and with the binary operation $\widetilde{\natural}_k$ when $(k,t,L)=(k,\tfrac{1}{2},2(1-k))$.\par

By a straightforward computation, one obtains that the relation
$$
-2t + 4kt + L = 1
$$
is equivalent to the condition that $F_{k,t,L}$ possesses joint homogeneity, namely,
\[
\lambda\, F_{k,t,L}(A,B) = F_{k,t,L}(\lambda A,\lambda B), \qquad (\lambda>0).
\]
Consequently, when $F_{k,t,L}$ is jointly homogeneous, the condition $k=\tfrac{1}{2}$ yields $L=1$, while $t=\tfrac{1}{2}$ yields $L=2(1-k)$. Thus, we present the following  {operator function with two parameters} that includes the spectral geometric means:

 \begin{dfn}
Let $A\in \mathcal{B}(H)_{++}$ and $B\in \mathcal{B}(H)_{+}$. For real parameters $k ,t \in (0,1)$, define
\[
F_{k ,t}(A,B):=
(A^{-1}\#_k  B)^{t }\,A^{1+2t-4kt}\,(A^{-1}\#_k  B)^{t }.
\]
In particular,
\[
F_{\frac{1}{2},t}(A,B) = A \natural_{t} B\quad \mbox{and} \quad F_{k,\frac{1}{2}}(A,B) = A \widetilde{\natural}_{k } B.
\]
We call $F_{\frac{1}{2},t}(A,B) = A \natural_{t} B$ the \textit{$(\frac{1}{2},t)$-spectral geometric mean}, $F_{k,\frac{1}{2}}(A,B) = A \widetilde{\natural}_{k } B$ the \textit{$(k,\frac{1}{2})$-spectral geometric mean}, and $F_{k,t}(A,B)$ the {\it $(k,t)$-spectral geometric mean}.
\end{dfn}

By the Riccati equation, it follows that $F_{k,t}(A,B)$ is the unique positive invertible solution $X$ to
\[
A^{-L} \# X = (A^{-1} \#_k B)^t,
\]
where $L=1+2t-4kt$.


\par
We present the basic properties of the $(k,t)$-spectral geometric mean $F_{k,t}$:
\begin{prp} \label{prp-Fp}
Let $A,B\in \mathcal{B}(H)_{++}$ and $k,t \in (0,1)$. Let $L=1+2t-4kt$. Then
\begin{enumerate}
\item If $A$ and $B$ commute, then $F_{k,t}(A,B) = A^{1-2kt}B^{2kt}$.
\item $F_{k,t}(\l A, \mu B) = \l^{1-2kt} \mu^{2kt} F_{k,t}(A,B)$ for $\l, \mu >0$.
\item $U^* F_{k,t}(A,B)U = F_{k,t}(U^*AU, U^*BU)$ for any unitary $U\in \mathcal{B}(H)$.
\item  $F_{k,t}(A,B)^{-1}=F_{k,t}(A^{-1},B^{-1})$.
\item  $A^{-L}\# F_{k,t}(A,B)=F_{1-k,t}(B,A)^{-1}\# B^L = (A^{-1} \#_k B)^t$. 
\item If $G_{k,t}=A^{-L}\# F_{k,t}(A,B)$, then $F_{k,t}(A,B)=G_{k,t}A^LG_{k,t}$ and $F_{1-k,t}(B,A)=G_{k,t}^{-1}B^L G_{k,t}^{-1}$.
\item If $2((1-k)A^{-1}+kB)^{-t}-A^L$ is invertible, then
\[
2((1-k)A+kB^{-1})^{-t}-A^{-L}\leq F_{k,t}(A,B) \leq \left[ 2\left((1-k)A^{-1}+kB\right)^{-t}-A^L\right]^{-1}.
\]
\end{enumerate}
\end{prp}

\begin{proof}
(1)--(3) easily follows from the definition of $F_{k,t}$.\\ 
For (4), since $(A\#B)^{-1}=A^{-1}\# B^{-1}$, we have (4).\\
For (5), it follows that
\begin{align*}
 & F_{1-k,t}(B,A)^{-1}\# B^L  = B^L \# F_{1-k,t}(B^{-1},A^{-1}) \quad \mbox{by (4)} \\
& = B^{\frac{L}{2}} \left( B^{-\frac{L}{2}} (B\#_{1-k} A^{-1})^t B^{-L} (B\#_{1-k} A^{-1})^t B^{-\frac{L}{2}} \right)^{\frac{1}{2}} B^{\frac{L}{2}}\\
& = B^{\frac{L}{2}} B^{-\frac{L}{2}} (B \#_{1-k} A^{-1})^t B^{-\frac{L}{2}} B^{\frac{L}{2}} \\
& = (A^{-1} \#_k B)^t \quad \mbox{by the transposition property of $\#_k$}.
\end{align*}
 {$A^{-L}\# F_{k,t}(A,B)=(A^{-1}\#_{k}B)^{t}$ can be shown by the similar way.}\\
For (6), it follows from the Riccati equation.\\
For (7), it follows from the arithmetic-geometric-harmonic mean inequality and $t\in (0,1)$ that
\begin{align*}
\left[ \frac{1}{2}\left( A^L+F_{k,t}(A,B)^{-1}\right) \right]^{-1} & \leq A^{-L} \# F_{k,t}(A,B) = (A^{-1} \#_k B)^t \quad \mbox{ {by (5)}}\\
& \leq \left( (1-k)A^{-1}+kB\right)^t 
\end{align*}
and hence
\begin{equation} \label{eq:fktab-1}
F_{k,t}(A,B)^{-1} \geq 2\left( (1-k)A^{-1}+kB\right)^{-t}-A^L. 
\end{equation}
Since RHS is invertible by assumption, taking the inverse on both sides, we have
\[
F_{k,t}(A,B) \leq \left[ 2\left( (1-k)A^{-1}+kB\right)^{-t}-A^L\right]^{-1}. 
\]
Since $F_{k,t}(A,B)$ is self-duality by (4), it follows from \eqref{eq:fktab-1} that
\[
F_{k,t}(A^{-1},B^{-1}) \geq 2\left( (1-k)A^{-1}+kB\right)^{-t}-A^L
\]
and replacing $A$ and $B$ by $A^{-1}$ and $B^{-1}$ respectively, we have
\[
F_{k,t}(A,B) \geq 2\left( (1-k)A+kB^{-1}\right)^{-t}-A^{-L}.
\]
\end{proof}

We shall say that $Z\in\mathcal{B}(H)$ is {\it positively similar} to $Y\in \mathcal{B}(H)$ if there exists $P\in \mathcal{B}(H)_{++}$ such that $Z=PYP^{-1}$. 
Gan-Tam \cite{GT} obtained a positively similarity relation between the geometric mean and $(\frac{1}{2},t)$-spectral geometric mean. We prove that the geometric mean and the $(k,t)$-spectral geometric mean are positively similar, which blackuced to the result of Gan-Tam \cite[Theorem 2.3]{GT} when $k=\frac{1}{2}$.\par

\begin{thm}
Let $A, B\in\mathcal{B}(H)_{++}$ and $k,t \in (0,1)$. Let $L=1+2t-4kt$. Then $A^L\# B^L$ is positively similar to 
\[
F_{1-k,t}(B,A)^{\frac{1}{2}}UF_{k,t}(A,B)^{\frac{1}{2}}
\]
for some unitary operator $U$.
\end{thm}

\begin{proof}
First of all, it is well-known  {(see, for example, \cite[Theorem 2.1]{FP})} that for $A,B\in \mathcal{B}(H)_{++}$, there exists a unitary operator $U$ such that
\begin{equation}
A\# B=A^{\frac{1}{2}}UB^{\frac{1}{2}}.
\label{eq:geometric mean with U}
\end{equation}
In fact, we put $U=(A^{\frac{-1}{2}}BA^{\frac{-1}{2}})^{\frac{1}{2}}A^{\frac{1}{2}}B^{-\frac{1}{2}}$. 

Let $G_{k,t}=(A^{-1}\#_{k}B)^{t}$. Then we have $G^{-1}_{k,t}F_{k,t}(A,B)G^{-1}_{k,t}=A^{L}$, and $|F_{k,t} {(A,B)}^{\frac{1}{2}}G_{k,t}^{-1}|=A^{\frac{L}{2}}$. Hence we have the polar decomposition $F_{k,t}(A,B)^{\frac{1}{2}}G_{k,t}^{-1}=V_{1}A^{\frac{L}{2}}$. By the same way, we have $F_{1-k,t}(B,A)^{\frac{1}{2}}G_{k,t}=V_{2}B^{\frac{L}{2}}$ for some unitary $V_{2}$, since $G_{k,t}=(A^{-1}\#_{k}B)^{t}=(B^{-1}\#_{1-k}A)^{-t}$.

Hence there exists a unitary operator $U$ such that
\begin{align*}
GF_{1-k,t}(B,A)^{\frac{1}{2}}UF_{k,t}(A,B)^{\frac{1}{2}}G^{-1}
& = B^{\frac{L}{2}}V_{2}^{*}UV_{1}A^{\frac{L}{2}}\\
& =B^{L}\# A^{L}=A^{L}\# B^{L},
\end{align*}
where the second equality holds by \eqref{eq:geometric mean with U}.
\end{proof}

 In the case of $k=\frac{1}{2}$, since $A \natural_t B$ satisfies the transposition property  {\cite[Proposition 4.2]{Lee and Lim}}, it follows that $F_{\frac{1}{2},t}(A,B) = F_{\frac{1}{2},1-t}(B,A)$. In the case of $k\not= \frac{1}{2}$, Dinh-Tam-Vuong in \cite[Remark 2.3]{Dinh-Tam-Vuong} showed that $F_{k,\frac{1}{2}}(A,B)$ does not satisfy the transposition property. Therefore, $F_{k,t}(A,B)$ does not satisfy the transposition property in general.\par


\par
\medskip

\section{Ando--Hiai property for $(k,t)$-spectral geometric mean}
%
Let
$F : \mathcal{B}(H)_{++} \times \mathcal{B}(H)_{+} \to \mathcal{B}(H)_{+}$
be a two-variable function. 
We say that $F$ has the \emph{Ando--Hiai property} for an exponent $q>0$ if it satisfies the implication: For $A, B\in \mathcal{B}(H)_{++}$,
\[
   F(A,B)\le I \ \Rightarrow\  F(A^q,B^q)\le I.
\]
If, in addition, $F$ satisfies properties similar to those of spectral geometric means,
namely, joint homogeneity and norm continuity in the second variable,
then the Ando--Hiai property can be characterized in terms of a norm inequality as follows.

\begin{prp}\label{prp-hom}
Let $q$ be a positive real number.
Suppose that a two-variable function $F: \mathcal{B}(H)_{++}\times \mathcal{B}(H)_{+}\to \mathcal{B}(H)_{+}$ satisfies the following conditions.
\begin{itemize}
  \item[(i)] $F(A,B)>0$ if $A,B>0$,
  \item[(ii)] $F(A,B)$ is norm continuous in the second variable,
  \item[(iii)] $F(A,B)$ is jointly homogeneous.
\end{itemize}
Then the following conditions are equivalent:
\begin{itemize}
\item[(1)] $F$ has the \emph{Ando--Hiai property} for $q$;
\item[(2)]
$
\|F(A^q,B^q)\|\le \|F(A,B)\|^q
\quad \text{for all } A>0 \text{ and } B\ge 0.
$
\end{itemize}
Here $\|\cdot\|$ denotes the operator norm.
\end{prp}
\begin{proof}
The implication $(2)\Rightarrow(1)$ is immediate.
We prove $(1)\Rightarrow(2)$.
We assume that $A>0$ and $B\geq 0 $.
Then for $\varepsilon >0$, $ B+\varepsilon I>0$, and by (i),
\[
\alpha_\varepsilon := \|F(A,B+\varepsilon I)\| >0.
\]
By the joint homogeneity of $F$ (iii), we have 
$$ F\left(\frac{A}{\alpha_{\varepsilon}},\frac{B+\varepsilon I}{\alpha_{\varepsilon}} \right)=\frac{F(A,B+\varepsilon I)}{\alpha_{\varepsilon}}\leq I.$$
Hence by (1),
we have
\[
\left\| F\!\left(\frac{A^q}{\alpha_\varepsilon^q},
\frac{(B+\varepsilon I)^q}{\alpha_\varepsilon^q}\right) \right\| \le 1.
\]
By (iii), it is equivalent to 
\[
\|F(A^q,(B+\varepsilon I)^q)\|
\le \|F(A,B+\varepsilon I)\|^q.
\]
By the norm continuity of $F$ in the second variable  (ii), letting $\varepsilon\downarrow 0$
yields the desiblack inequality.
\end{proof}

We remark that the $(k,t)$-spectral geometric mean $F_{k,t}(A,B)$ satisfies the conditions (i)--(iii) in Proposition \ref{prp-hom}. ((ii) can be proven by the functional calculus.)

\color{black}

We show that a two-variable operator function $F_{k,t}$ for $k,t\in (0,1)$ satisfies the Ando-Hiai property for some $q$. To show this, we consider $2\times 2$ matrices $A_{x,y}$ and $B$ defined below at first. Then we give a necessary condition of $q$ such that Ando--Hiai inequality holds. Next, we give a sufficient condition of $q$ such that Ando--Hiai inequality holds by using operator theoretic approach.
%
%
%

\begin{prp}\label{prp: 2--by--2}
For real numbers $x,y>0$, set
\[
 A_{x,y}:=\begin{pmatrix} x+y & y-x \\ y-x & x+y \end{pmatrix}, 
 \qquad 
 B:=\begin{pmatrix} 1 & 0 \\ 0 & 0 \end{pmatrix}.
\]
Suppose $q>0$ and ${1\over 2}>t (1-k )$ \ (resp. ${1\over 2}<t (1-k )$). 
If  {Proposition \ref{prp-hom} (2)} 
holds for all $x,y>0$, then necessarily $0< q \le 1$ \ (resp. $q\ge 1$).
\end{prp}

\medskip
To prove Proposition \ref{prp: 2--by--2}, we recall the following lemma, 
whose proof can be found in \cite[Lemma 4]{Hoa-Osaka-Wada}.

\begin{lem} \label{lem-axy}
For positive real numbers $x,y>0$ and $k\in (0,1)$, the following equation holds:
\[
\begin{pmatrix} x+y & y-x \\ y-x & x+y \end{pmatrix}
\#_{k }
\begin{pmatrix} 1 & 0 \\ 0 & 0 \end{pmatrix}
=
\left(\frac{4xy}{x+y}\right)^{1-k }
\begin{pmatrix} 1 & 0 \\ 0 & 0 \end{pmatrix}.
\]
\end{lem}

\begin{proof}[Proof of Proposition \ref{prp: 2--by--2}]
By Proposition \ref{prp-hom}, we shall show
\begin{equation}\label{a-h}
 \|F_{k,t}(A_{x,y}^q, B^q)\|
 \;\le\;
 \|F_{k,t}(A_{x,y}, B)\|^q
\end{equation}
 {holds for all  $x,y>0$ and only when $q\in (0,1]$.}
Put $L=1+2t-4kt$. By a straightforward computation,

\begin{align*}
  A_{x,y}^{q} 
  &= \frac{1}{2}\begin{pmatrix} (2x)^{q}+(2y)^{q} & (2y)^{q}-(2x)^{q} \\ (2y)^{q}-(2x)^{q} & (2x)^{q}+(2y)^{q} \end{pmatrix} 
   = A_{\tfrac{(2x)^q}{2}, \tfrac{(2y)^q}{2}}, \\
  A_{x,y}^{-1} 
  &= \frac{1}{4xy}\begin{pmatrix} x+y & x-y \\ x-y & x+y \end{pmatrix} 
   = \frac{1}{4xy}A_{yx}. 
\end{align*}

It follows from Lemma~\ref{lem-axy} that 
\[
 (A_{x,y}^{-1}\#_{k } B)^{t} = \left(\frac{1}{x+y}\right)^{(1-k )t}B, 
 \]
 \[
 F_{k ,t}(A_{x,y},B)=
   \left(
   \frac{(2x)^L+(2y)^L}{2(x+y)^{2t(1-k )}}\right)B.
\]
Define
\begin{equation}
\varphi_{k ,t,L}(x,y)
   := {\|F_{k,t}(A_{x,y},B)\|=} \frac{(2x)^L+(2y)^L}{2(x+y)^{2t(1-k )}}.
\label{eq:star}
\end{equation}
Then inequality (\ref{a-h}) is equivalent to

$$\varphi_{k ,t,L}\!\left(\frac{(2x)^q }{2},\frac{(2y)^q }{2}\right)
\le
\varphi_{k ,t,L}(x,y)^q .$$
By setting $x=1$ and letting $y\rightarrow 0$, we obtain
\[
2^{2t(1-q)(1-k )+qL-1} \le 2^{q(L-1)},
\]
which is equivalent to
\[
q(1-2t(1-k ))\le 1-2t(1-k ).
\]
This completes the proof.
\end{proof}


 {We put $k=\frac{1}{2}$ and $t=\frac{1}{2}$ in Proposition \ref{prp: 2--by--2}, we can obtain the following corollaries, respectively.}

\begin{cor}
Suppose $t\in (0,1)$ and $q>0$. If the  {implication}
\begin{equation*}\label{ando-hiai-norm} 
A\natural_t B\le I \Rightarrow A^q \natural_t B^q\le I
 \end{equation*}
holds for all  {$A\in \mathcal{B}(H)_{++}$ and $B\in \mathcal{B}(H)_{+}$}, then $0< q \le 1$.
\end{cor}

\begin{cor}\label{pre-ando-hiai}
Suppose $k\in (0,1)$ and $q>0$. If the  {implication}
\begin{equation*} \label{ando-hiai-norm2} 
 A\widetilde{\natural}_k B\le I \Rightarrow A^q \widetilde{\natural}_k B^q\le I
 \end{equation*}
holds for all  {$A\in \mathcal{B}(H)_{++}$ and $B\in \mathcal{B}(H)_{+}$}, then $0< q \le 1$.
\end{cor}

Next, we give a sufficient condition of $q$ for which 
 $F_{k,t}$ satisfies 
 the Ando-Hiai property 
for all $A,B\in \mathcal{B}(H)_{++}$.

\begin{thm}\label{key1}
Let $A, B\in \mathcal{B}(H)_{++}$, and $k,t\in (0,\frac{1}{2}]$. Put  $L=1+2t-4kt$. 
Then the implication
$$F_{k,t}(A,B)\le I \ \Rightarrow \ F_{k,t}(A^q,B^q)\le I$$
holds for all $q$ in the range
\[
0<q \le
\left(\frac{1}{2kt}-\frac{1-k}{Lk}\right)^{-1}
=\frac{2ktL}{1-2kt} \leq 1.
\]
\end{thm}

To prove Theorem~\ref{key1}, we need the following result.

\begin{thm}[Grand Furuta Inequality {\cite[Theorem 1.1]{Furuta1995}}]\label{thm:GFI}
Let $0\leq B\leq A$ such that $A$ is invertible. Then %
$$ [A^{\frac{r}{2}}(A^{\frac{-\alpha}{2}}B^{p}A^{\frac{-\alpha}{2}})^{s}A^{\frac{r}{2}}]^{\frac{1-\alpha+r}{(p-\alpha)s+r}}\leq A^{1-\alpha+r}
$$
holds for all $p,s\geq 1$, $\alpha\in [0,1]$ and $r\geq \alpha$.
\end{thm}

\begin{proof}[Proof of Theorem~\ref{key1}]
Since $0<t\le \tfrac{1}{2}$ and $0<k\leq \frac{1}{2}$, we have $L\geq 1$. From the assumption, we have
$$(A^{-1}\#_k B)^{2t}\le A^{-L}.$$
By the Grand Furuta Inequality, 
\[
\left[
A^{-\tfrac{rL}{2}}
\left\{
A^{\tfrac{\alpha L}{2}}
(A^{-1}\#_k B)^{2tp}
A^{\tfrac{\alpha L}{2}}
\right\}^s
A^{-\tfrac{rL}{2}}
\right]^{\frac{1-\alpha+r}{(p-\alpha)s+r}}\leq A^{-L(1-\alpha+r)}
\]
holds for all $p,s\geq 1$, $\alpha\in [0,1]$, and $r\geq \alpha$.
Now, putting 
\[
(p,s,\alpha,r):=\left(\tfrac{1}{2t}, \tfrac{1}{k}, \tfrac{1}{L}, \tfrac{1}{L}\right),
\]
the above inequality becomes
\begin{equation}
B^{
\left(\frac{1}{2tk}-\frac{1-k}{Lk}\right)^{-1}}
\le A^{-L}.
\label{eq:see}
\end{equation}
Setting 
\[
x:=
\left(\frac{1}{2kt}-\frac{1-k}{Lk}\right),
\]
this can be written as
\[
B^{1/x} \le A^{-L}.
\]
Since $0<qx\le 1$, we obtain
\[
B^q = B^{\tfrac{1}{x}\cdot qx} \le A^{-Lqx}.
\]
Therefore, it follows from $0<2t\leq 1$ that
\[
(A^{-q}\#_k B^q )^{2t}\le
(A^{-q}\#_k A^{-Lqx})^{2t}
=A^{-Lq}.
\]
This implies 
\begin{align*}
\NORM{F_{k,t}(A^q,B^q)} & = \NORM{(A^{-q} \#_k B^q)^t A^{qL} (A^{-q} \#_k B^q)^t} = \NORM{A^{\frac{qL}{2}}(A^{-q} \#_k B^q)^{2t} A^{\frac{qL}{2}}} \\
& \leq \NORM{A^{\frac{qL}{2}}(A^{-q} \#_k A^{-qxL})^{2t} A^{\frac{qL}{2}}} =1
\end{align*}
and so
\[
F_{k,t}(A^q,B^q)\le I.
\]
Next, from the evident relation
\[
2t(Lk+1-k)=2t(kL+(1-k)\cdot 1)\le 2tL\le L,
\]
we obtain
\[
2kt\le \frac{Lk}{Lk+1-k}.
\]
Rearranging this inequality gives
\[
1\ge \left(\frac{1}{2kt}-\frac{1-k}{Lk}\right)^{-1} \ (\ge q).
\]
\end{proof}
\begin{rmk}
The above theorem also holds for the binary operation $F_{k,t,L}$ with no joint homogeneity, for $L \ge 1$.
Indeed, in the proof above, the detailed definition of $L$ is not used; only the fact that $L \ge 1$ is requiblack.
\end{rmk}

%
%

As  {applications} of Theorem~\ref{key1}, we show that the $(\frac{1}{2},t)$  {and $(k,\frac{1}{2})$}-spectral geometric  {means have} the Ando-Hiai property under some restricted conditions:
 
\begin{thm}\label{thm-spg}
Let $A, B\in \mathcal{B}(H)_{++}$,  {and $t\in (0,1)$.} 
Then
$$A\natural_{t}B\leq I \quad \Longrightarrow \quad
A^{q} \natural_{t}B^{q} \leq I\quad\text{for all $0\leq q\leq \min\{\frac{t}{1-t},\frac{1-t}{t}\} $.}$$
Especially, in the case of $t=\frac{1}{2}$,
$$A\natural B\leq I \quad \Longrightarrow \quad
A^{q} \natural B^{q} \leq I\quad\text{for all $0\leq q\leq 1$}.$$
\end{thm}

\begin{proof}
The case of $0<t\leq \tfrac{1}{2}$ follows immediately from Theorem~\ref{key1} by putting $k=\frac{1}{2}$.

For $\frac{1}{2}\leq t<1$, i.e., 
$0<1-t\leq \frac{1}{2}$, we obtain
$$B\natural_{1-t}A\leq I \quad \Longrightarrow \quad
B^{q} \natural_{1-t}A^{q} \leq I$$
 {for all $0\leq q\leq \min \{\frac{1-t}{1-(1-t)}, \frac{1-(1-t)}{1-t} \}=\min \{\frac{1-t}{t}, \frac{t}{1-t} \}$.}

Since $A\natural_{t}B=B\natural_{1-t}A$, the proof is completed.
\end{proof}

\begin{rmk}
We do not know whether the Ando-Hiai property for the $(\frac{1}{2},t)$-spectral geometric mean holds for  {$\min\{\frac{t}{1-t},\frac{1-t}{t}\} <q< 1$.} We searched for a counterexample using a computer but were unable to find one.
\end{rmk}

\begin{pbm}
  Let $A,B \in \mathcal{B}(H)_{++}$ and  {$t\in (0,1)\setminus\{\frac{1}{2}\}$}. 
Then does the implication
$$ A\natural_{t}B \le I \ \Rightarrow \ A^{q}\natural_{t}B^{q} \le I
$$
hold for all $0<q\leq 1$?
\end{pbm}

\medskip


Next, we show the Ando-Hiai property for $(k,\frac{1}{2})$-spectral  {geometric} mean under some restricted conditions.

\begin{thm}
Let $A, B \in \mathcal{B}(H)_{++}$. 
 If $0<k\leq \tfrac{1}{2}$, then
\begin{equation}\label{ando-hiai2}
A \,\widetilde{\natural}_{k}\, B\leq I \quad \Longrightarrow\quad
A^{q} \,\widetilde{\natural}_{k}\, B^{q} \leq I
\quad\text{for all \textcolor{black}{$0< q \le 2k$}.}
\end{equation}
\end{thm}

\begin{proof}
If we put $t=\frac{1}{2}$ and $0<k\leq \frac{1}{2}$ in Theorem~\ref{key1}, then we have \eqref{ando-hiai2} for all $0<q\leq 2k$.
\end{proof}

It is not clear whether the statement (\ref{ando-hiai2}) holds when $0<k\leq \tfrac{1}{2}$ and $1>q>2k$. By contrast, in the opposite regime where $k>\frac{1}{2}$, one cannot expect favorable properties.


\begin{thm}\label{counter}
There exist  {$A\in \mathcal{B}(H)_{++}$, $B\in \mathcal{B}(H)_{+}$}
 {and} $k\in (\frac{1}{2},1)$ such that \eqref{ando-hiai2} does not hold for any exponent $q \in (0,1)$.
\end{thm}

In order to show this  {theorem}, we require the following lemma.  
For arbitrary $x,y,L>0$, define
\[
h_{x,y}(L):=
\log \left(\frac{(2x)^L+ (2y)^L}{2(x+y)^{\frac{L}{2}}}\right)-\frac{L}{4}\log 4xy.
\]

\begin{lem} \label{lem-hxy}
For any distinct $x, y > 0$, there exists a constant $L_{x,y} \in(0,1)$ such that 
\[
   h_{x,y}(L) < 0 \qquad \text{for all } L \in (0, L_{x,y}).
\]

\end{lem}

\begin{proof}
This follows immediately from the facts that
\[
\left.\frac{dh_{x,y}}{dL}\right|_{L=0}=
{1\over 4}\log \left( \frac{4xy}{(x+y)^2}\right) <0,
\]
$h_{x,y}(0)=0$ and $h_{x,y}(1)={1\over 4}\log\frac{(x+y)^2}{4xy} >0$.
\end{proof}

\begin{proof}[Proof of Theorem~\ref{counter}]
In the case of $t=\frac{1}{2}$ and $\frac{1}{2}<k< 1$, we have $L=2-2k$ and $0\leq L<1$. For the $2 \times 2$ matrices $A_{x,y}$ and $B$ treated in Proposition~\ref{prp: 2--by--2}, we shall show that for some $\frac{1}{2}<k< 1$ and $q_k\in (0,1)$, 
\begin{equation}
\NORM{A_{x,y}^q  \,\widetilde{\natural}_{k}\, B^q}
\leq \NORM{A_{x,y}  \,\widetilde{\natural}_{k}\, B}^q
\label{eq:AH norm inequality}
\end{equation}
does not hold for any $q\in (0,q_k)$. We consider the function $\varphi_{k,t,L}(x,y)$ which is appeablack  {in \eqref{eq:star}}
%
as follows.
\[
\varphi_{k,\frac{1}{2},L}(x,y)  {=\|A_
{x,y}\tilde{\natural}_{k}B\|}= \frac{(2x)^{L} + (2y)^{L}}{2(x+y)^{\frac{L}{2}}}.
\]
Then  {we have}
\[
\NORM{A_{x,y}^q  \,\widetilde{\natural}_{k}\, B^q} = \NORM{A_{\frac{(2x)^q}{2},\frac{(2y)^q}{2}} \,\widetilde{\natural}_{k}\, B} =  \varphi_{k,\frac{1}{2},L}\left(\frac{(2x)^q}{2}, \frac{(2y)^q}{2}\right),
\]
and 
\[
\NORM{A_{x,y}  \,\widetilde{\natural}_{k}\, B}^q=\varphi_{k,\frac{1}{2},L}(x,y)^q.
\]
Then, for some $L,q_L\in (0,1)$, we aim to prove
\[
\varphi_{k,\frac{1}{2},L}(x,y)^q- \varphi_{k,\frac{1}{2},L}\!\left( \frac{(2x)^q}{2}, \frac{(2y)^q}{2}\right) <0\qquad \mbox{for all $q\in (0,q_L)$}.
\]
Set
\[
g_{L,x,y}(q):=\log \varphi_{k,\frac{1}{2},L}(x,y)^q- \log \varphi_{k,\frac{1}{2},L}\!\left( \frac{(2x)^q}{2}, \frac{(2y)^q}{2}\right).
\]
It follows from Lemma~\ref{lem-hxy} that for any distinct $x, y > 0$, 
there exists $L_{x,y}\in(0,1)$ such that
\[
\left.\frac{dg_{L,x,y}}{dq}\right|_{q=0}=h_{x,y}(L)<0  \qquad  (L\in (0,L_{x,y})).
\]
Since $g_{L,x,y}(0)=0$, the desiblack claim follows.

Let $k\in (\frac{1}{2},1)$ and $q_{ {L}}\in (0,1)$ such that \eqref{eq:AH norm inequality}
does not hold for all $q\in (0,q_{ {L}} {)}$. Assume that there exists $q_{0} \in [q_{ {L}},1)$ such that \eqref{eq:AH norm inequality} holds. Then we have
$$
\NORM{A_{x,y}  \,\widetilde{\natural}_{k}\, B}\geq 
\NORM{A_{x,y}^{q_{0}}  \,\widetilde{\natural}_{k}\, B^{q_{0}}}^{\frac{1}{q_{0}}}
\geq 
\NORM{A_{x,y}^{q_{0}^{2}}  \,\widetilde{\natural}_{k}\, B^{q_{0}^{2}}}^{\frac{1}{q_{0}^{2}}}
\geq \cdots \geq
\NORM{A_{x,y}^{q_{0}^{n}}  \,\widetilde{\natural}_{k}\, B^{q_{0}^{n}}}^{\frac{1}{q_{0}^{n}}}
$$
for all $n=1,2,...$. Since $q_{0}\in (0,1)$, 
$q_{0}^{n}\in (0,q_{ {L}})$  {for sufficiently large $n$}, and it is a contradiction.
Hence, the proof is completed.
\end{proof}


\par
\medskip

\section{Two variable version of Ando--Hiai property}
In this section, we present a two variable version of the Ando--Hiai property for $(\frac{1}{2},t)$-spectral geometric mean.
Let $A,B\in \mathcal{B}(H)_{++}$. Consider
the following inequality:
\begin{equation}\label{ando-hiai_ex}
A \na_t B \le I
\;\Rightarrow\;
A^{r} \na_{\frac{rt}{s(1-t)+rt}} B^{s} \le I,
\end{equation}
where $r,s\geq 0$ and $t\in(0,1)$.  
Since $\na_t$ is jointly homogeneous for every $t\in(0,1)$, the inequality above is
equivalent to the following norm inequality:
\begin{equation}
\|A^{r} \na_{\frac{rt}{s(1-t)+rt}} B^{s}\|
\;\le\;
\|A \na_t B\|^{\,r\!\left(1-\frac{rt}{s(1-t)+rt}\right)
+ s\!\left(\frac{rt}{s(1-t)+rt}\right)}.
\label{eq: norm-one sided}
\end{equation}
Based on this fact, we investigate the admissible ranges of the exponents
$r$ and $s$ for which the inequality \eqref{ando-hiai_ex} holds.

%
\begin{thm}
Suppose $t\in (0,1)$. If the inequality \eqref{ando-hiai_ex} holds for all  {$A\in \mathcal{B}(H)_{++}$ and $B\in \mathcal{B}(H)_{+}$}, then $r,s\in (0,1]$.
\end{thm}

\begin{proof}
For the $2\times 2$ matrices $A(=A_{x,y})$ and $B$ treated in
Proposition~\ref{prp: 2--by--2}, let us compute both sides of the inequality  {\eqref{eq: norm-one sided}}.
Set
\[
\alpha := \frac{rt}{s(1-t)+rt}.
\]
Then, by  {\eqref{eq:star}},
\[
\|A_{x,y} \na_t B\|  {=\varphi_{\frac{1}{2},t,1}(x,y)}= (x+y)^{1-t}.
\]
From the norm inequality just above, we obtain
\begin{align*}
\|A^{r} \na_{\frac{rt}{s(1-t)+rt}} B^{s}\|
&=
\|A_{ \frac{(2x)^r}{2},\ \frac{(2y)^r}{2} } \na_\alpha B\|
=
\left(
\frac{(2x)^r}{2} + \frac{(2y)^r}{2}
\right)^{1-\alpha} \\
&\le
\|A_{x,y} \na_t B\|^{\,r(1-\alpha)+s\alpha}
=
(x+y)^{(1-t)(\,r(1-\alpha)+s\alpha\,)} .
\end{align*}

Letting $x=1$ and $y\to 0$, we obtain
\[
2^{(r-1)(1-\alpha)} \le 1,
\]
which implies $r\le 1$.  
By interchanging the roles of $A$ and $B$, the same argument yields $s\le 1$.  
Hence we conclude that, in order for the inequality \eqref{ando-hiai_ex} to hold,
the exponents $r$ and $s$ must belong to the interval $(0,1]$.
\end{proof}


We now examine the relationship among the parameters $t$, $r$, and $s$.
\begin{thm} \label{thm-1}
Let $A, B\in\mathcal{B}(H)_{++}$.
\begin{itemize}
\item[(1)] If $0<t\leq \frac{1}{2}$, then
\[
A \na_t B \leq I \quad \Longrightarrow \quad A^r \na_{\frac{rt}{s(1-t)+rt}} B^s \leq I
\]
for $0<r\leq \frac{s(1-t)}{t}$ with $0<s\leq \frac{t}{1-t}\ (\leq 1)$.
\item[(2)] If $\frac{1}{2}\leq t {<} 1$, then
\[
A \na_t B \leq I \quad \Longrightarrow \quad A^r \na_{\frac{rt}{s(1-t)+rt}} B^s \leq I
\]
for $0<s\leq \frac{rt}{1-t}$ with $0<r\leq \frac{1-t}{t}\ (\leq 1)$.
\end{itemize}
In particular, 
\[
A\na B \leq I \quad \Longrightarrow \quad  A^r \na_{\frac{r}{r+s}} B^s \leq I 
\]
for all $r,s\in (0,1]$.
\end{thm}

It is obvious that the case $r=s$ in Theorem~\ref{thm-1} is just the Ando-Hiai inequality of the $(\frac{1}{2},t)$-spectral geometric mean  (Theorem~\ref{thm-spg}).\par

To prove Theorem~\ref{thm-1}, we need the following lemma, which is a one-sided version of Theorem~\ref{thm-1}:

\begin{lem} \label{thm-2}
Let $A, B\in\mathcal{B}(H)_{++}$.
\begin{itemize}
\item[(1)] If $0<t\leq \frac{1}{2}$, then
\[
A \na_t B \leq I \quad \Longrightarrow \quad A \na_{\frac{t}{s(1-t)+t}} B^s \leq I \qquad \mbox{for $0<s\leq \frac{t}{1-t}$.}
\]
\item[(2)] If $\frac{1}{2}\leq t  {<} 1$, then
\[
A \na_t B \leq I \quad \Longrightarrow \quad A^r \na_{\frac{rt}{rt+1-t}} B \leq I \qquad \mbox{for $0<r\leq \frac{1-t}{t}$.}
\]
\end{itemize}
\end{lem}

\begin{proof}
 {Proof of (2).}
Suppose that $\frac{1}{2}\leq t {<} 1$. By putting $k=\frac{1}{2}$ (then $L=1$) in the proof of Theorem \ref{key1}, $A\natural_{t}B\leq I$ implies
$$ B^{\frac{t}{1-t}}\leq A^{-1}$$
by \eqref{eq:see}.
Since $\frac{1-t}{t} \leq 1$, we have
$$ B \leq A^{-\frac{1-t}{t}} {.}$$
Hence we have 
\begin{align*}
\NORM{A^r \na_{\frac{rt}{rt+1-t}} B} & = \NORM{A^{\frac{r}{2}}(A^{-r}\s B)^{\frac{2rt}{rt+1-t}}A^{\frac{r}{2}}} \\
& \leq \NORM{A^{\frac{r}{2}}(A^{-r}\s A^{\frac{t-1}{t}})^{\frac{2rt}{rt+1-t}}A^{\frac{r}{2}}} = 1,
\end{align*}
because $\frac{2rt}{rt+1-t}\leq 1 \Leftrightarrow \ rt \leq 1-t$. Therefore, we have (2).

 {Proof of (1).}
Suppose that $0<t\leq \frac{1}{2}$. Since $\frac{1}{2}\leq 1-t {<} 1$, it follows from (2) and the transposition property of the $(\frac{1}{2},t)$-spectral  {geometric} mean $\na_t$ that
\[
A \na_t B = B \na_{1-t} A \leq I \quad \Longrightarrow \quad B^s \na_{\frac{s(1-t)}{s(1-t)+t}} A = A \na_{\frac{t}{s(1-t)+t}} B^s \leq I
\]
for $0<s\leq \frac{t}{1-t}$. Hence we have (1).
\end{proof}

\begin{proof}[Proof of Theorem~\ref{thm-1}]  
 {Proof of  {(1)}.} Suppose that $0<t\leq \frac{1}{2}$. It follows from (2) of Lemma~\ref{thm-2} and $\frac{1}{2}\leq 1-t  {<} 1$ that
\[
A\na_t B = B \na_{1-t} A \leq I \quad \Longrightarrow \quad B^s \na_{\frac{s(1-t)}{s(1-t)+t}} A = A \na_{\frac{t}{s(1-t)+t}} B^s \leq I
\]
for $0<s\leq \frac{t}{1-t}$. Since $0<s\leq \frac{t}{1-t}$ implies $\frac{1}{2}\leq \frac{t}{s(1-t)+t}  {<} 1$, by using  {Lemma~\ref{thm-2} (2)} again, we have
\[
A^r \na_{\frac{rt}{rt+s(1-t)}} B^s \leq I
\]
for $0<r \leq \frac{1-\frac{t}{s(1-t)+t}}{\frac{t}{s(1-t)+t}} = \frac{s(1-t)}{t}$, and so we have (1).

 {Proof of  {(2)}.}
Suppose that $\frac{1}{2}\leq t  {<} 1$. It follows from $0<1-t\leq \frac{1}{2}$ and (1) that
\[
A\na_t B = B \na_{1-t} A \leq I \quad \Longrightarrow \quad B^s \na_{\frac{s(1-t)}{s(1-t)+rt}} A^r = A^r \na_{\frac{rt}{s(1-t)+rt}} B^s \leq I
\]
for $0<s\leq \frac{rt}{1-t}$ with $0<r\leq \frac{1-t}{t}$, and we have (2).
\end{proof}

\begin{pbm}
  Let $A, B\in \mathcal{B}(H)_{++}$, and  {$t\in (0,1)\setminus \{\frac{1}{2}\}$}.  
Then does  {\eqref{ando-hiai_ex}} hold for all $r,s\in (0,1]$?
\end{pbm}



\section{Log-majorization}

 In this section and beyond, we will deal with matrices rather than operators. Let ${\Bbb M}_n={\Bbb M}_n({\Bbb C})$ be the algebra of $n\times n$ complex matrices, and ${\Bbb P}_n={\Bbb P}_n({\Bbb C})$ the algebra of $n\times n$ complex positive definite matrices, and denote the matrix absolute value of any $A \in {\Bbb M}_n$ by $|A|=(A^*A)^{\frac{1}{2}}$. For positive semidefinite $A, B$ let us write $A \prec_{w(\log)}B$ and refer to the weak log-majorization if
\[
\prod_{i=1}^k \l_i(A) \leq \prod_{i=1}^k \l_i(B) \qquad \mbox{for $k=1,2,\ldots, n$,}
\]
 {where $\l_1(A)\geq \l_2(A)\geq \cdots \geq \l_n(A)$ and $\l_1(B)\geq \l_2(B)\geq \cdots \geq \l_n(B)$ are the eigenvalues of $A$ and $B$ respectively.} Further the log-majorization $A\prec_{(\log)} B$ means that $A \prec_{w(\log)} B$ and the equality holds for $k=n$ in the above, i.e., 
\[
\prod_{i=1}^n \l_i(A) = \prod_{i=1}^n \l_i(B) \qquad \mbox{i.e., $\det A = \det B$.}
\]
 It is known that for positive semidefinite $A, B$,
\begin{equation}
  {A \prec_{(\log)} B \ \Longrightarrow}\  A \prec_{w(\log)} B\  \Longrightarrow\  \UIN{A} \leq \UIN{B}
\label{eq:log-norm relation}
\end{equation}
for any unitarily invariant norm $\UIN{\cdot}$. \par
 For each $A\in {\Bbb M}_{n}$ and $i=1,\ldots,n$, let $C_i(A)$ denote the $i$-fold antisymmetric tensor power of $A$. See \cite{MO} for details. By the Binet-Cauchy theorem \cite[P. 123]{Zhang}, for every positive semidefinite $A$
\begin{equation} \label{eq:BC}
\prod_{i=1}^j \l_i(A) = \l_1(C_{j}(A)) \qquad \mbox{for $j=1,2,\ldots,n$}
\end{equation}
and it is known that $\l_1(C_j(A)) \leq \l_1(C_j(B))$ for all $j=1,\ldots,n$ if and only if $A \prec_{w(\log)} B$ for positive  {semidefinite} $A$ and $B$.\par
By Proposition~\ref{key1} and the anti-symmetric tensor technique \eqref{eq:BC}, we can get the following log-majorization  {relation}:
\begin{thm} \label{thm-lmF}
Let $A, B\in {\Bbb P}_n$ and $k,t\in (0,\frac{1}{2}]$. Let $L=1+2t-4kt$. Then
$F_{k,t}(A^q,B^q)^{\frac{1}{q}} \prec_{(\log)} F_{k,t}(A^p, B^p)^{\frac{1}{p}}$ for all $0<q \leq \frac{2ktL}{1-2kt} p$.
\end{thm}

\begin{proof}
$\frac{2ktL}{1-2kt}\in (0,1]$ is shown in Theorem \ref{key1}.
Let $k,t\in (0,\frac{1}{2} {]}$ and $0<r\leq \frac{2ktL}{1-2kt}$. It follows that
for $i=1,\ldots,n$,
\begin{align*}
C_i(F_{k,t}(A^{r}, B^{r})) & =F_{k,t}(C_i(A)^{r}, C_i(B)^{r}), \\
C_i(F_{k,t}(A, B)^{r}) & =F_{k,t}(C_i(A), C_i(B))^{r}.
\end{align*}
Also,
\[
\det(F_{k,t}(A^{r}, B^{r})) = (\det(A)^{(k-1){r}}\det(B)^{k{r}})^{2t} \det(A)^{{r}L} = \det(F_{k,t}(A, B)^{r}).
\]
Hence it suffices by \eqref{eq:BC} to show that
\[
\l_1(F_{k,t}(A^{r}, B^{r})) = \NORM{F_{k,t}(A^{r}, B^{r})} \leq \NORM{F_{k,t}(A, B)^{r}} = \l_1(F_{k,t}(A, B)^{r}).
\]
By Proposition~\ref{prp-hom}, we may prove that 
\[
F_{k,t}(A, B)\leq I \quad {\Longrightarrow} \quad F_{k,t}(A^{r}, B^{r}) \leq I \quad \mbox{for all $0<{r}\leq \frac{2ktL}{1-2kt}$}.
\]
This is just Proposition~\ref{key1}, that is, we have $F_{k,t}(A^{r},B^{r}) \prec_{(\log)} F_{k,t}(A,B)^{r}$.
Put $r=\frac{q}{p}\in (0,\frac{2ktL}{1-2kt}]$. We have
\[
F_{k,t}(A^{\frac{q}{p}},B^{\frac{q}{p}}) \prec_{(\log)} F_{k,t}(A,B)^{\frac{q}{p}}.
\]
Replacing $A$ and $B$ by $A^p$ and $B^p$, the proof is completed.
\end{proof}

There is a gap in the proof of \cite[Theorem 3.3]{GT}, so we cannot say that the Ando-Hiai type inequality for the $(\frac{1}{2},t)$-spectral geometric mean holds for $0<q\leq 1$. However, as an application of Theorem~\ref{thm-lmF}, we show that it holds under certain restricted conditions, and it is a partial solution of the result of Gan-Tam \cite[Theorem 3.3]{GT}:

\begin{cor} \label{thm-ah}
Let $A, B\in{\Bbb P}_n$ and  {$t \in (0,1)$}.
\[
 (A^q\natural_t B^q)^{\frac{1}{q}}  \ \prec_{(\log)}\ (A^p \natural_t B^p)^{\frac{1}{p}} \quad\text{for all $ {0<} q\leq \min\{\frac{1-t}{t}, \frac{t}{1-t}\}p$,}
\]
%
In particular, if $t=\frac{1}{2}$, then
\[
(A^q \natural B^q)^{\frac{1}{q}} \ \prec_{(\log)}\ (A^{p} \natural B^p)^{\frac{1}{p}} \quad\text{for all  {$0< q\leq p$}.}
\]
\end{cor}
\begin{proof}
 {Proof is similar to the proof of Theorem \ref{thm-spg}.}
If we put $k=\frac{1}{2}$ in Theorem~\ref{thm-lmF}, then we have $L=1$ and by the transposition property of $\natural_t$, we have this corollary.
\end{proof}

If we put $t=\frac{1}{2}$ in Theorem~\ref{thm-lmF}, then we have the following corollary for the $(k,\frac{1}{2})$-spectral geometric mean:

\begin{cor} \label{thm-kah}
Let $A, B\in {\Bbb P}_n$ and $0<k\leq \frac{1}{2}$. Then
\[
(A^q \widetilde{\natural}_k B^q)^{\frac{1}{q}} \ \prec_{(\log)}\ (A^{p}\widetilde{\natural}_k B^{p})^{\frac{1}{p}} \quad \mbox{for all  {$0< q\leq 2kp$.}}
\]
\end{cor}

We note that Corollaries \ref{thm-ah} and \ref{thm-kah} have been proven in \cite[Proposition 3.10]{Hpreprint2025_2}. However, the domain of $q$ in Corollary \ref{thm-kah} strictly contains the domain of $q$ in \cite[Proposition 3.10]{Hpreprint2025_2}.
 We show the following Lie-Trotter formula for the $(k,t)$-spectral geometric mean, which includes one for the $(k,\frac{1}{2})$-spectral geometric mean \cite[Theorem 3.1]{Dinh-Tam-Vuong}:

\begin{thm} \label{thm-LTF}
Let $A, B\in\mathcal{B}(H)$ be self-adjoint and $k,t\in (0,1]$. Then
\[
\lim_{p\to 0} F_{k,t}(\exp[pA],\exp[pB])^{\frac{1}{p}} = \exp \left((1-2kt)A+2kt B\right)
\]
in the norm topology.
\end{thm}
\begin{proof}
 {It} follows from a similar method of \cite[Theorem 3.1]{Dinh-Tam-Vuong}. Indeed, for $p\in (0,1)$, we express $p=\frac{1}{m+s}$, where $m\in {\Bbb N}$ and $s\in (0,1)$. Put
\[
X(p) = F_{k,t}(\exp[pA], \exp[pB]) \ \mbox{and} \ Y(p)=\exp[p((1-2kt)A+2kt B)],
\]
and $L=1+2t-4kt$. It follows from \eqref{eq:se} that
\begin{align*}
\NORM{\exp[-pA] \#_k \exp[pB]} & \leq \NORM{\exp[-p(1-k)A+pkB]} \\
& \leq \exp[p(1-k)\NORM{A}+pk\NORM{B}],
\end{align*}
and hence as $pm\leq 1$, we have
\begin{align*}
\NORM{X(p)}^m & = \NORM{(\exp[-pA] \#_k \exp[pB])^t \exp[pLA] (\exp[-pA] \#_k \exp[pB])^t}^m \\
& \leq \NORM{\exp[-pA] \#_k \exp[pB]}^{2tm} \NORM{\exp[pLA]}^m \\
& \leq \exp[pm((1+4t-6kt)\NORM{A}+2kt \NORM{B})] \\
& \leq \exp[(1+4t-6kt)\NORM{A}+2kt \NORM{B}] <+\infty.
\end{align*}
Therefore, it follows that
\[
\NORM{X(p)^{\frac{1}{p}}-X(p)^m} \leq \NORM{X(p)}^m \NORM{X(p)^s-I} \to 0\quad \mbox{as $p\to 0$.}
\]
Similarly, we have 
\[
\NORM{Y(p)^m - Y(p)^{\frac{1}{p}}} \to 0 \quad \mbox{as $p\to 0$.}
\]
Next, we have
\begin{align*}
X(p) & = (\exp[-pA] \#_k \exp[pB])^t \exp[pLA] (\exp[-pA] \#_k \exp[pB])^t \\
& = \left[ I+pt(-(1-k)A+k B)+o(p)\right] \left[ I+p(1+2t-4kt)A+o(p)\right]\\
& \qquad \qquad \qquad \qquad \times \left[ I+pt(-(1-k)A+k B)+o(p)\right] \\
& = I+p((1-2kt)A+2kt B)+o(p).
\end{align*}
Since $Y(p)=I+p((1-2kt)A+2ktB)+o(p)$, we have $\NORM{X(p)-Y(p)}\leq cp^2$ for some constant $c$, and so we have
\begin{align*}
\NORM{X(p)^m-Y(p)^m} & = \NORM{\sum_{j=1}^{m-1} X(p)^{m-1-j}(X(p)-Y(p))Y(p)^j} \\
& \leq mM^{m-1}\NORM{X(p)-Y(p)} \\
& \leq mM^{m-1}cp^2 \leq \frac{m}{m+s} M^{m-1} cp \to 0 \quad \mbox{as $p\to 0$},
\end{align*}
where $M=\max \{ \NORM{X(p)}, \NORM{Y(p)}\}$ and $M^{m-1}<\infty$. Therefore, we have
\begin{align*}
& \NORM{X(p)^{\frac{1}{p}}-Y(p)^{\frac{1}{p}}} \\
& \leq \NORM{X(p)^{\frac{1}{p}}-X(p)^m}+\NORM{X(p)^m-Y(p)^m}+\NORM{Y(p)^m-Y(p)^{\frac{1}{p}}} \\
& \to 0 \quad \mbox{as $p\to 0$.}
\end{align*}
\end{proof}
Combining Theorems \ref{thm-lmF} and \ref{thm-LTF}, we obtain the following  relation:
\begin{thm} \label{thm-lgF}
Let $A, B\in {\Bbb P}_n$ and $k,t\in (0,\frac{1}{2}]$. 
Then
\[
\exp\left((1-2kt)\log A + 2kt \log B\right) \prec_{(\log)} F_{k,t}(A^p,B^p)^{\frac{1}{p}} \quad \mbox{for all $p>0$}.
\]
\end{thm}

We show the log-majorization relations among geometric means:
\begin{thm} \label{thm-lmiF}
Let $A, B\in {\Bbb P}_n$. 
\begin{itemize}
\item[(1)] If $k,t\in (0,\frac{1}{2}]$, then 
\begin{align*}
A \#_{2kt} B & \prec_{(\log)} \exp [ (1-2kt)\log A + 2kt \log B ]  \\
& \prec_{(\log)} \left( A^{\frac{(1-2kt)p}{2}}B^{2ktp}A^{\frac{(1-2kt)p}{2}}\right)^{\frac{1}{p}} \prec_{(\log)} \left( A^{\frac{p}{2}}B^{\frac{2ktp}{1-2kt}}A^{\frac{p}{2}}\right)^{\frac{1-2kt}{p}}\\
& \prec_{(\log)} F_{k,t}(A^p, B^p)^{\frac{1}{p}} \prec_{(\log)} \left( A^{\frac{(1-2kt)p}{4kt}}B^pA^{\frac{(1-2kt)p}{4kt}}\right)^{\frac{2kt}{p}}
\end{align*}
for all $p>0$.
\item[(2)] If $k,t\in [\frac{1}{2},1)$, 
then
\begin{align*}
& B \#_{2(1-k)(1-t)} A \\
& \prec_{(\log)} \exp [ 2(1-k)(1-t)\log A + (1-2(1-k)(1-t)) \log B ]  \\
& \prec_{(\log)} \left( A^{(1-k)(1-t)p}B^{(1-2(1-k)(1-t))p}A^{(1-k)(1-t)p}\right)^{\frac{1}{p}} \\
& \prec_{(\log)} \left( A^{\frac{(1-k)(1-t)p}{1-2(1-k)(1-t)}}B^{p}A^{\frac{(1-k)(1-t)p}{1-2(1-k)(1-t)}}\right)^{\frac{1-2(1-k)(1-t)}{p}}\\
& \prec_{(\log)} F_{1-k,1-t}(B^p, A^p)^{\frac{1}{p}} \prec_{(\log)} \left( A^{\frac{p}{2}}B^{\frac{(1-2(1-k)(1-t))p}{2(1-k)(1-t)}}A^{\frac{p}{2}}\right)^{\frac{2(1-k)(1-t)}{p}}
\end{align*}
for all $p>0$.
\end{itemize}
\end{thm}

\begin{proof}
(1): Since $0<2kt\leq \frac{1}{2}$, the first and second log-majorization relations 
follow from \cite{AH1994}. For the third log-majorization relation, 
it suffices to prove that
\[
A^{\frac{p}{2}}B^{\frac{2ktp}{1-2kt}}A^{\frac{p}{2}}\leq I \quad \Longrightarrow \quad A^{\frac{(1-2kt)p}{2}}B^{2ktp}A^{\frac{(1-2kt)p}{2}} \leq I.
\]
If $A^{\frac{p}{2}}B^{\frac{2ktp}{1-2kt}}A^{\frac{p}{2}}\leq I$ or  $B^{\frac{2ktp}{1-2kt}}\leq A^{-p}$, then it follows from  {$\frac{1}{2}\leq 1-2kt < 1$} and L\"{o}wner-Heinz theorem that $B^{2ktp}\leq A^{-p(1-2kt)}$, and thus $A^{\frac{(1-2kt)p}{2}}B^{2ktp}A^{\frac{(1-2kt)p}{2}} \leq I$.\par
For the fourth log-majorization relation, 
it suffices to prove that
\[
F_{k,t}(A^p, B^p) \leq I \quad \Longrightarrow \quad A^{\frac{p}{2}}B^{\frac{2ktp}{1-2kt}}A^{\frac{p}{2}} \leq I.
\]
If $F_{k,t}(A^p, B^p) \leq I$ or $(A^{-p} \#_k B^p)^{2t} \leq A^{-pL}$, then it follows from \eqref{eq:see} 
that $B^{\frac{2ktpL}{1-2kt}} \leq A^{-pL}$ and so $B^{\frac{2ktp}{1-2kt}} \leq A^{-p}$ by $L\geq 1$. Hence we have $A^{\frac{p}{2}}B^{\frac{2ktp}{1-2kt}}A^{\frac{p}{2}} \leq I$.\par
 For the fifth log-majorization relation, 
it suffices to prove that
\[
A^{\frac{(1-2kt)p}{4kt}}B^pA^{\frac{(1-2kt)p}{4kt}} \leq I \quad \Longrightarrow \quad F_{k,t}(A^p, B^p) \leq I.
\]
If $A^{\frac{(1-2kt)p}{4kt}}B^pA^{\frac{(1-2kt)p}{4kt}} \leq I$ or $B^p \leq A^{\frac{(2kt-1)p}{2kt}}$, then we have
\[
(A^{-p} \#_k B^p)^{2t} \leq (A^{-p} \#_k A^{\frac{(2kt-1)p}{2kt}})^{2t} = A^{-pL}
\]
and $F_{k,t}(A^p, B^p) \leq \NORM{F_{k,t}(A^p, B^p)} = \NORM{A^{\frac{pL}{2}}(A^{-p} \#_k B^p)^{2t} A^{\frac{pL}{2}}} \leq I$, and so we have the desiblack inequality (1).\par
For (2), if  {$k,t\in [\frac{1}{2}, 1)$}, then $1-k,1-t\in (0,\frac{1}{2}]$ and thus (2) follows from (1).
\end{proof}

Gan-Tam in \cite[Theorem 3.6]{GT} showed the following log-majorization relation for the $(\frac{1}{2},t)$-spectral geometric mean: For $t \in [0,1]$,
\begin{equation} \label{eq:GT36}
\left( A^{\frac{(1-t)p}{2}}B^{tp}A^{\frac{(1-t)p}{2}}\right)^{\frac{1}{p}} \prec_{(\log)} \left( A^p \na_t B^p\right)^{\frac{1}{p}}
\end{equation}
for all  {$A,B\in \mathcal{B}(H)_{++}$} and $p>0$. In particular,
\[
A \#_t B \prec_{(\log)} \exp[(1-t)\log A+t \log B] \prec_{(\log)} A^{\frac{1-t}{2}}B^tA^{\frac{1-t}{2}} \prec_{(\log)} A \na_t B.
\]
As an application of Theorem~\ref{thm-lmiF}, if we put $k=\frac{1}{2}$, then we obtain the following refinement of \eqref{eq:GT36}:
\begin{cor} \label{cor-na}
Let $A,B\in {\Bbb P}_n$. 
\begin{itemize}
\item[(1)] If $0<t \leq \frac{1}{2}$, then 
\begin{align*}
A \#_{t} B & \prec_{(\log)} \exp [ (1-t)\log A + t \log B ]  \\
& \prec_{(\log)} \left( A^{\frac{(1-t)p}{2}}B^{tp}A^{\frac{(1-t)p}{2}}\right)^{\frac{1}{p}} \prec_{(\log)} \left( A^{\frac{p}{2}}B^{\frac{tp}{1-t}}A^{\frac{p}{2}}\right)^{\frac{1-t}{p}}\\
& \prec_{(\log)} (A^p \na_t B^p)^{\frac{1}{p}} \prec_{(\log)} \left( A^{\frac{(1-t)p}{2t}}B^pA^{\frac{(1-t)p}{2t}}\right)^{\frac{t}{p}}
\end{align*}
for all $p>0$.
\item[(2)] If  {$\frac{1}{2}\leq t < 1$}, then
\begin{align*}
A \#_{t} B & \prec_{(\log)} \exp [ (1-t)\log A + t \log B ]  \\
& \prec_{(\log)} \left( A^{\frac{(1-t)p}{2}}B^{tp}A^{\frac{(1-t)p}{2}}\right)^{\frac{1}{p}} \prec_{(\log)} \left( A^{\frac{(1-t)p}{2t}}B^{p}A^{\frac{(1-t)p}{2t}}\right)^{\frac{t}{p}}\\
& \prec_{(\log)} (A^p \na_t B^p)^{\frac{1}{p}} \prec_{(\log)} \left( A^{\frac{p}{2}}B^{\frac{tp}{1-t}}A^{\frac{p}{2}}\right)^{\frac{1-t}{p}}
\end{align*}
for all $p>0$.
\end{itemize}
\end{cor}

\begin{rmk}
In the case of $p=1$, the last relation in (1) of Corollary~\ref{cor-na} is shown in \cite[Theorem 4.8]{FS2025}.
\end{rmk}

As an application of Theorem~\ref{thm-lmiF}, if we put $t=\frac{1}{2}$, then we obtain the following log-majorization relation for the $(k,\frac{1}{2})$-spectral geometric mean:
\begin{cor}\label{cor: tilde}
Let $A,B\in {\Bbb P}_n$. 
\begin{itemize}
\item[(1)] If $0<k \leq \frac{1}{2}$, then 
\[
\left( A^{\frac{p}{2}}B^{\frac{kp}{1-k}}A^{\frac{p}{2}}\right)^{\frac{1-k}{p}} \prec_{(\log)} (A^p \widetilde{\natural}_k B^p)^{\frac{1}{p}} \prec_{(\log)} \left( A^{\frac{(1-k)p}{2k}}B^pA^{\frac{(1-k)p}{2k}}\right)^{\frac{k}{p}}
\]
for all $p>0$.
\item[(2)] If  {$\frac{1}{2}\leq t < 1$}, then
\[
\left( A^{\frac{(1-k)p}{2k}}B^{p}A^{\frac{(1-k)p}{2k}}\right)^{\frac{k}{p}}
 \prec_{(\log)} (A^p \widetilde{\natural}_k B^p)^{\frac{1}{p}} \prec_{(\log)} \left( A^{\frac{p}{2}}B^{\frac{kp}{1-k}}A^{\frac{p}{2}}\right)^{\frac{1-k}{p}}
\]
for all $p>0$. 
\end{itemize}
\end{cor}

 {The second log-majorization relation in Corollary \ref{cor: tilde} (2) has been already shown in \cite[Corollary 3.9]{JKT2025}.}

\section{Norm inequalities}

In this section, we discuss norm inequalities for spectral geometric means: Firstly, by Theorem~\ref{thm-lgF}, we have the following norm inequality:
\begin{thm} \label{thm-NF}
Let $A, B\in {\Bbb P}_n$ and $k,t\in (0,\frac{1}{2}]$. 
Then
\begin{equation} \label{eq:eF}
\UIN{\exp\left((1-2kt)\log A + 2kt \log B\right)} \leq \UIN{F_{k,t}(A^p,B^p)^{\frac{1}{p}}} \ 
\end{equation}
for all $p>0$ and any unitarily invariant norm $\UIN{\cdot}$. Moreover, for each $p>0$, there exists a sequence \textcolor{black}{$\{ p_m\}$ such that $p_m \downarrow 0$ as $m \to \infty$} and 
\[
\UIN{ F_{k,t}(A^{\textcolor{black}{p_m}}, B^{\textcolor{black}{p_m}})^{\frac{1}{\textcolor{black}{p_m}}}} \ \searrow \ \UIN{\exp\left((1-2kt)\log A+2kt \log B\right)} \quad \mbox{\textcolor{black}{as $m \to \infty$.}}
\]
\end{thm}

\begin{proof}
Put $L=1+2t-4kt$.\par
Assume that $k,t\in (0,\frac{1}{2}]$. 
For $0<q\leq \frac{2ktL}{1-2kt}p (\leq p)$, it follows from Theorem~\ref{thm-lmF}  {and \eqref{eq:log-norm relation}} that
\[
\UIN{F_{k,t}(A^q, B^q)^{\frac{1}{q}}}  \ \leq\ \UIN{F_{k,t}(A^p,  B^p)^{\frac{1}{p}}} 
\]
and as $q\to 0$, it follows from Theorem~\ref{thm-LTF} that
\[
\UIN{\exp\left((1-2kt)\log A+2kt\log B\right)} \leq \UIN{F_{k,t}(A^p, B^p)^{\frac{1}{p}}} \quad \mbox{for all  {$0< \frac{2ktL}{1-2kt}p$}}.
\]
Therefore, we have the desiblack inequality \eqref{eq:eF}.\par

For any $p>0$, put $p_0=p$ and $p_1=\frac{2ktL}{1-2kt} p_0$. Then $p_1 \leq p_0$ and  it follows from Theorem~\ref{thm-lmF}  {and \eqref{eq:log-norm relation}} that
\[
\UIN{F_{k,t}(A^{p_1}, B^{p_1})^{\frac{1}{p_1}}}  \ \leq\ \UIN{F_{k,t}(A^{p_0}, B^{p_0})^{\frac{1}{p_0}}}.
\]
Similarly, we put \textcolor{black}{$p_m=\frac{2ktL}{1-2kt} p_{m-1}$ for $m=1,2,\ldots$}. Then \textcolor{black}{$p_m \downarrow 0$ as $m \to \infty$}, and 
\[
\UIN{F_{k,t}(A^{\textcolor{black}{p_m}}, B^{\textcolor{black}{p_m}})^{\frac{1}{\textcolor{black}{p_m}}}} \ \searrow \ \UIN{\exp\left((1-2kt)\log A+2kt\log B\right)} \quad \mbox{\textcolor{black}{as $m \to \infty$.}}
\]
\end{proof}

\begin{rmk}
By Theorem~\ref{thm-NF}, we have the following norm inequality for the $(k,t)$-spectral geometric mean: Let $k,t\in (0,\frac{1}{2}]$. 
Then
\[
\UIN{A \#_{2kt} B} \leq \UIN{\exp\left((1-2kt)\log A+2kt \log B\right)} \leq \UIN{F_{k,t}(A,B)}.
\]
In the case of $t=\frac{1}{2}$,
\[
\UIN{A \#_k B} \leq \UIN{\exp\left((1-k)\log A+k \log B\right)} \leq \UIN{A \widetilde{\natural}_k B}.
\]
\end{rmk}

Gan-Tam showed in \cite[Corollary 3.10]{GT} that if $A,B\in {\Bbb P_{n}}$ and $t\in [0,1]$, then
\begin{equation} \label{eq:uen-1}
\UIN{\exp\left((1-t)\log A+t\log B\right)} \leq \UIN{(A^p\natural_t B^p)^{\frac{1}{p}}}\quad \mbox{for all $p>0$,}
\end{equation}
for any unitarily invariant norm $\UIN{\cdot}$. Moreover, 
\begin{equation} \label{eq:uen-2}
\UIN{(A^p \natural_t B^p)^{\frac{1}{p}}} \ \searrow \ \UIN{\exp\left((1-t)\log A+t\log B\right)}\quad \mbox{as $p\to 0$}.
\end{equation}
\medskip

There is a gap in the proof of the Ando-Hiai inequality \cite[Corollary 3.10]{GT} for the spectral geometric mean $\natural_t$, so it is unclear whether \eqref{eq:uen-2} is correct. However, \eqref{eq:uen-1} of the argument is correct.\par
 Thus, using Theorem~\ref{thm-NF}, we show the following results with a slight modification of results by Gan-Tam.

\begin{thm}\label{thm: unitarily invariant norm 1}
Let $A,B\in {\Bbb P}_n$ and  {$t\in (0,1)$}. Then
\begin{equation} \label{eq:1}
\UIN{\exp\left((1-t)\log A+t\log B\right)} \leq \UIN{(A^p \natural_t B^p)^{\frac{1}{p}}} \quad \mbox{for all $p>0$}
\end{equation}
for any unitarily invariant norm $\UIN{\cdot}$. 
Moreover, for each $p>0$, there exists a sequence \textcolor{black}{$\{ p_m\}$ such that $p_m \downarrow 0$ as $m\to \infty$} and
\begin{equation} \label{eq:2}
\UIN{(A^{\textcolor{black}{p_m}} \natural_t\ B^{\textcolor{black}{p_m}})^{\frac{1}{\textcolor{black}{p_m}}}} \ \searrow \  \UIN{\exp\left((1-t)\log A+t\log B\right)} \ \mbox{\textcolor{black}{as $m \to \infty$.}}
\end{equation}
\end{thm}

\begin{proof}
\eqref{eq:1} follows from Corollary \ref{cor-na}  {and \eqref{eq:log-norm relation}}. 
In the case of $0<t\leq \frac{1}{2}$, if we put $k=\frac{1}{2}$ in Theorem~\ref{thm-NF}, then we have \eqref{eq:2}.\par
In the case of  {
$\frac{1}{2}\leq t< 1$}, since $0<1-t\leq \frac{1}{2}$ and the $(\frac{1}{2},t)$-spectral geometric mean satisfies the transposition property, we have \eqref{eq:2}.
\end{proof}

\begin{rmk}
In particular, if $t=\frac{1}{2}$, then it follows from Corollary~\ref{thm-ah}  {and \eqref{eq:log-norm relation}} that $\UIN{(A^p \natural B^p)^{\frac{1}{p}}}$ decrease to $\UIN{\exp\left(\frac{\log A+\log B}{2}\right)}$ as $p\to 0$. 
\end{rmk}


%
Gan--Tam \cite{GT} showed that
\begin{small}
\begin{equation}
\begin{split}
\UIN{A \sharp_t B} & \leq \UIN{\exp\left((1-t)\log A+t \log B\right)} \leq \UIN{\left( B^{\frac{tp}{2}}A^{(1-t)p}B^{\frac{pt}{2}} \right)^{\frac{1}{p}}} \\
& \leq \UIN{(A^p\natural_t B^p)^{\frac{1}{p}}}
\end{split}
\label{eq:GT}
\end{equation}
\end{small}
hold for $t\in [0,1]$ and $p>0$. 

By Corollary~\ref{cor-na}  {and \eqref{eq:log-norm relation}}, we show a refinement of \eqref{eq:GT} due to Gan-Tam:

\begin{thm} \label{thm-uab}
Let $A, B\in {\Bbb P}_{n}$.
\begin{itemize}
\item[(1)] If $0<t\leq \frac{1}{2}$, then
\begin{small}
\[
\UIN{\left(A^{\frac{(1-t)p}{2}}B^{pt}A^{\frac{(1-t)p}{2}}\right)^{\frac{1}{p}}} \leq \UIN{(A^{\frac{p}{2}}B^{\frac{pt}{1-t}}A^{\frac{p}{2}})^{\frac{1-t}{p}}}\leq \UIN{(A^{p} \na_t B^{p})^{\frac{1}{p}}} \leq \UIN{\left(A^{\frac{(1-t)p}{2t}}B^{p}A^{\frac{(1-t)p}{2t}}\right)^{\frac{t}{p}}}
\]
\end{small}
for all $p>0$.
\item[(2)] If  {$\frac{1}{2}\leq t< 1$}, then
\begin{small}
\[
\UIN{\left(A^{\frac{(1-t)p}{2}}B^{pt}A^{\frac{(1-t)p}{2}}\right)^{\frac{1}{p}}} \leq \UIN{(A^{\frac{(1-t)p}{2t}}B^{p}A^{\frac{(1-t)p}{2t}})^{\frac{t}{p}}} \leq \UIN{(A^{p} \na_t B^{p})^{\frac{1}{p}}} \leq \UIN{(A^{\frac{p}{2}}B^{\frac{pt}{1-t}}A^{\frac{p}{2}})^{\frac{1-t}{p}}}
\]
\end{small}
for all $p>0$.
\end{itemize}
\end{thm}

\begin{rmk}
Replacing $A \na_t B$ by $A \widetilde{\natural}_k B$ in Theorem~\ref{thm-uab}, we have the same results.
\end{rmk}

\begin{pbm}
For any $A,B\in \mathcal{B}(H)_{++}$, is there any order relation between $A \na_t B$ and $A \widetilde{\natural}_k B$?
\end{pbm}

 {In \cite[Corollary 3.9]{JKT2025}, a partial answer is given  {when}  $k=t\in [\frac{1}{2}, 1]$.}

\section{Alternative means}
\subsection{Week log-majorization for alternative means}

Let $A\in \mathcal{B}(H)_{++}$ and $B\in \mathcal{B}(H)_{+}$.
In \cite{DFKC2026}, as a generalization of the spectral geometric mean, the {\it alternative mean} $A\hat{\sigma}_{f}B$ is defined as follows.
$$ A\hat{\sigma}_{f}B=f(A^{-1}\sharp B)Af(A^{-1}\sharp B), $$
where $f$ is a positive operator monotone function defined on  {$[0,\infty)$} satisfying $f(1)=1$. It interpolates spectral geometric mean $A\natural_{t}B$ (if $f(x)=x^{t}$) and the Wasserstein mean $A\lozenge_{t}B$ (if $f(x)=1-t +t x$). In \cite[Theorem 4.4]{GK2024}, a week log-majorization relation between spectral geometric and Wasserstein means is discussed. However the proof seems to be incorrect. In this section, we give a generalization of it. We denote the singular values of  {$A\in \mathbb{M}_{n}$} by $s_1(A)\geq s_2(A) \geq \cdots \geq s_n(A) \geq 0$.

\begin{thm}\label{thm:eigenvalue inequality}
Let  {$A, B\in \mathbb{M}_{n}$ be positive semidefinite} such that $A$ is invertible, and let $f,g$ be  {non-negative} operator monotone functions defined on $ {[}0,\infty)$ satisfying $f(1)=g(1)=1$. If $f(x)\leq g(x)$ holds for all  {$x\geq 0$}, then 
$$ s_{j}(A\hat{\sigma}_{f}B)\leq s_{j}(A\hat{\sigma}_{g}B) $$
holds for all $j=1,2,...,n$. 
\end{thm} 

\begin{proof}
It is enough to show that there exists a unitary matrix $U$ such that 
$$  A\hat{\sigma}_{f}B  \leq  U^{*} A\hat{\sigma}_{g}B U. $$
From the definition of alternative mean, there exists unitary matrices $V_{f}$ and $V_{g}$ such that 
$$ A\hat{\sigma}_{f}B=V^{*}_{f}A^{\frac{1}{2}}f(A^{-1}\sharp B)^{2} A^{\frac{1}{2}}V_{f}$$
and 
$$ A\hat{\sigma}_{g}B=V^{*}_{g}A^{\frac{1}{2}}g(A^{-1}\sharp B)^2 A^{\frac{1}{2}}V_{g}.$$
By the assumption $f(x)\leq g(x)$ for all  {$x\geq 0$}, we have $ f(x)^{2}\leq g(x)^{2}$, and hence 
\begin{align*}
A\hat{\sigma}_{f}B
& =
V^{*}_{f}A^{\frac{1}{2}}f(A^{-1}\sharp B)^{2} A^{\frac{1}{2}}V_{f}\\
& \leq 
V^{*}_{f}A^{\frac{1}{2}}g(A^{-1}\sharp B)^{2} A^{\frac{1}{2}}V_{f}\\
& =
V^{*}_{f}V_{g}A\hat{\sigma}_{g}BV^{*}_{g}V_{f}.
\end{align*}
\end{proof}

 {As a direct consequence of Theorem \ref{thm:eigenvalue inequality}, we obtain the following.}

\begin{cor}
Let  {$A, B\in \mathbb{M}_{n}$ be positive semidefinite} such that $A$ is invertible, and let $f,g$ be  {non-negative} operator monotone functions defined on $ {[}0,\infty)$. If  {$f(x)\leq g(x)$} holds for all  {$x\geq 0$}, then 
$$ A\hat{\sigma}_{f}B \prec_{w(\log)} A\hat{\sigma}_{g}B. $$
\end{cor}

 {Moreover by putting $f(x)=x^t$ and $g(x)=1-t+tx$, we have the following corollary.}

\begin{cor}
Let  {$A, B\in \mathbb{M}_{n}$ be positive semidefinite} such that $A$ is invertible. Then 
$$ A\natural_{t}B \prec_{w(\log)} A\lozenge_{t}B $$
holds for all  {$t\in (0,1)$}.
\end{cor}
\par\medskip

\subsection{Ando--Hiai Property}
As in the case of the spectral geometric mean, 
some alternative means $\hat{\sigma}_f$ also satisfy the Ando--Hiai property: 
\begin{equation}\label{alternative Ando-Hiai}
A \hat{\sigma}_f B \le I 
\;\Rightarrow\;
A^q \hat{\sigma}_f B^q \le I,
\qquad (A,B>0)
\end{equation}
for some $q>0$. 
Without knowing the explicit form of the function $f$, it is difficult to determine the exact set of exponents $q$ for which the Ando--Hiai property holds.  
However, we can give a rough estimate on the possible range of such exponents as follows:

\begin{prp}
Let $f$ be a non-negative normalized operator monotone function on $[0,\infty)$. 
If $f$ is non-trivial (that is, $f(t) \neq 1$ and $f(t) \neq t$), then
\[
\{\, q>0 \mid
\hat{\sigma}_f \text{ has the Ando--Hiai property for } q \,\}
\subseteq (0,1].
\]
\end{prp}
\begin{proof}
Assume that $\hat{\sigma}_f$ has the Ando--Hiai property for an exponent $q>0$. 
Then  {we have the following implications.}

\begin{align*}
A\hat{\sigma_{f}}B\leq I  
& \Longrightarrow 
A^{q}\hat{\sigma_{f}}B^{q}\leq I
 \Longrightarrow 
A^{q^2}\hat{\sigma_{f}}B^{q^2}\leq I\\
& \Longrightarrow 
\cdots 
 \Longrightarrow 
A^{q^n}\hat{\sigma_{f}}B^{q^n}\leq I,
\end{align*}
i.e, \color{black}the same property holds for every exponent $q^n$ $(n\ge 1)$. 
Since
\begin{align*}
(\l A) \hat{\sigma}_f (\l B) & = f((\l A)^{-1} \# (\l B))(\l A) f((\l A)^{-1} \# (\l  {B})) \\
& = f(A^{-1} \# B)(\l A)f(A^{-1} \# B) = \l (A \hat{\sigma}_f B)
\end{align*}
for $\l>0$, the binary operation $\hat{\sigma}_f$ is jointly homogeneous and hence we obtain the following norm inequality  {by Proposition \ref{prp-hom}}:
\begin{equation}\label{alternative Ando-Hiai2}
\|A^{q^n} \hat{\sigma}_f B^{q^n}\|
\;\le\;
\|A \hat{\sigma}_f B\|^{q^n},
\qquad (A>0,\, B\ge 0,\, n\ge 1).
\end{equation}

For the $2\times 2$ matrices $A(=A_{x,y})$ and $B$ consideblack in 
Proposition~\ref{prp: 2--by--2}, we put
\[
\alpha_{x,y} := f(1/\sqrt{x+y}), 
\qquad 
 {\beta := f(0),}
\]
and define
\[
\varphi_{f}(x,y) := \|A_{x,y} \hat{\sigma}_f B\|
=
\left\|
\begin{pmatrix} 
\alpha_{x,y} & 0 \\ 
0 & \beta 
\end{pmatrix} 
\begin{pmatrix} 
x+y & y-x \\ 
y-x & x+y 
\end{pmatrix} 
\begin{pmatrix} 
\alpha_{x,y}	 & 0 \\ 
0 & \beta 
\end{pmatrix} 
\right\|
\]
for $x,y>0$.  {We remark that $\beta<1$ because $f$ is a non-trivial operator monotone function.}

By the same argument as in Proposition~\ref{prp: 2--by--2}, the inequality \eqref{alternative Ando-Hiai2} yields the following estimate:
\begin{equation*}
\varphi_{f}\left(\frac{(2x)^{q^n}}{2}, \frac{(2y)^{q^n}}{2}\right) 
\le \varphi_{f}(x,y)^{q^n}, 
\qquad (n\ge 1).
\end{equation*}
Taking the limit $(x,y)\to(1,0)$, we obtain
\begin{equation}\label{tmp2}
2^{q^n-1} \left(f^2\left(\frac{1}{\sqrt{2^{q^n-1}}}\right) 
+ \beta^{2}\right)
\le (1+\beta^2)^{q^n}.
\end{equation}

We first consider the case $\beta>0$.  
Then \eqref{tmp2} can be rewritten as
\[
\frac{f^2\left(\frac{1}{\sqrt{2^{q^n-1}}}\right) + \beta^{2}}{2}
\le 
\left(\frac{1+\beta^2}{2}\right)^{q^n}.
\]
Suppose $q>1$. Then letting $n\to\infty$ yields
\[
\beta^{2}
= \lim_{n\to \infty} 
\frac{f^2\left(\frac{1}{\sqrt{2^{q^n-1}}}\right) + \beta^{2}}{2} 
\le 
\lim_{n\to \infty} 
\left(\frac{1+\beta^2}{2}\right)^{q^n} 
= 0 \quad  {\mbox{(by $\beta<1$),}}
\]
which contradicts the assumption $\beta >0$. 
Thus, we must have $q \le 1$.

In contrast, we next consider the remaining case $f(0)=0$.
Taking $n=1$ in \eqref{tmp2}, we have
\[
2^{q-1} f^2\left(\frac{1}{\sqrt{2^{q-1}}}\right) \le 1.
\]
Since the function 
$t \mapsto \bigl(\sqrt{t}f(1/\sqrt{t})\bigr)^2$
is an operator monotone function on $(0,\infty)$, 
it follows that $2^{q-1}\le 1$, 
that is, $q \le 1$.
\end{proof}





{\bf Acknowledgments.} 
The work of Y. Seo and S. Wada was supported in part by Japan Society for the Promotion of Science
(grant-in-aid for scientific research) [grant numbers (C)23K03249 and (C)23K03141, respectively]. T. Yamazaki was supported by the INOUE ENRYO Memorial Grant, Toyo University.



\end{document}